\documentclass[fleqn]{article}
\usepackage{aism-2e}
\usepackage{amsmath,amssymb}
\newtheorem{theorem}{\sc Theorem}
\newtheorem{lemma}{\sc Lemma}

\newtheoremrm{definition}{\sc Definition}[thrmctr]
\newtheoremrm{assumption}{\sc Assumption}[thrmctr]
\newtheoremrm{condition}{\sc Condition}[thrmctr]
\newtheoremrm{remark}{\it Remark\/}                  
\newtheoremrm{example}{\it Example\/}
\newtheoremrm{case}{\it Case\/}
\newtheoremrmno{proof}{\sc Proof}
\newtheoremrmno{procedure}{\sc Procedure}
\newtheoremrmno{note}{\it Note\/}
\newtheoremrmno{fact}{\it Fact\/}
\makeatletter
\@addtoreset{corollary}{section}
\@addtoreset{proposition}{section}
\@addtoreset{assumption}{section}
\@addtoreset{condition}{section}
\frompage{1}\topage{28}
\publishyear{2008}

 \newcommand{\cE}{\mathcal{E}}

 \newcommand{\cL}{\mathcal{L}}

 \newcommand{\bC}{\mathbb{C}}
 \newcommand{\bR}{\mathbb{R}}
 \newcommand{\bN}{\mathbb{N}}
 \newcommand{\bZ}{\mathbb{Z}}
 
 \newcommand{\erf}{\mathrm{erf}}

\begin{document}
\sloppy
\openup1\jot
\addtolength{\textheight}{2.5cm}

\title{
The Exact Distribution of the Sample Variance from Bounded Continuous Random Variables}

\author{
T. Royen
}
\affiliation{
Fachhochschule Bingen, University of applied sciences,\\
 Berlinstrasse 109, D--55411 Bingen, Germany\\
 E--mail: royen@fh-bingen.de
}


\abstract{For a sample of absolutely bounded i.i.d. random variables with a continuous density  the cumulative distribution function of the sample variance is represented by a univariate integral over a Fourier--series. If the density is a polynomial or a trigonometrical polynomial the coefficients of this series are simple finite terms containing only the error function, the exponential function and powers. In more general cases --- e.g. for all beta densities --- the coefficients are given by some series expansions. The method is generalized to positive semi--definite quadratic forms of bounded independent, but not necessarily identically distributed random variables if the form matrix differs from a diagonal matrix $D > 0$ only by a matrix of rank 1.
} 

\date{11.9.2008}
\maketitle

\section*{\large 1. \ Introduction }
\label{section:1}
There is a lot of literature on the approximation of distributions of sample statistics for small samples. With the increasing technical availability of transcendental functions it becomes easier to compute exact representations  for many distributions. For the investigation of the accuracy of any proposed approximations to a distribution it would also be helpful to have an exact representation and not to rely only on Monte Carlo methods.

In particular, for the exact distribution of the sample variance $s^2$ from non--normal random variables only a few scattered results seem to exist. E.g. Mudholkar and Trivedi (1981) recommend transformations of the Wilson--Hilferty type to approximate the distribution of $s^2$. Former work is cited in this paper. In Royen (2007a), (2007b) several different representations are found for the cumulative distribution function (cdf) of $s^2$ from a gamma parent distribution. The first paper also contains an orthogonal series for the cdf of Greenwood's statistic (the square sum of "spacings") as a by--product. In the second paper a representation of the cdf of $s^2$ from a uniform distribution is given, which is a univariate integral over a Fourier series. The method for this special case is generalized in the underlying paper.

Let be
\begin{equation}
\label{eq:1}
Q =Q_n =(n-1)s^2 = \sum^{n}_{i=1} (X_i - \overline{X})^2,
\end{equation}
where $X_1,\ldots,X_n$ are bounded i.i.d. random variables with a continuous density $f$ satisfying some mild regularity conditions. Without loss of generality the interval $(0,1)$ (or $(-1,1)$) can be chosen as the support of $f$. By geometrical considerations on the $n$--unit cube it can be shown that
\begin{equation}
\label{eq:2}
q:= \sup Q_n = \Bigg\{ \begin{matrix} n/4 & , \ n \ \mathrm{even}\\
                                (n^2-1)/(4n) & , \ n \ \mathrm{odd} \end{matrix} \Bigg\}.
\end{equation}

An estimation for the order of magnitude of the cdf $F_Q(x) $ of $Q$ for small $x$ in the following section suggests for the density $f_Q$ of $Q$
\begin{equation}
\label{eq:3}
	\lim_{x\to 0} f_Q(x) = \lim_{x\to q} f_Q(x) = 0,
\end{equation}
already for small values $n$ if $f(x)$ does not increase "too rapidly" for $x$ tending to 0 or to 1. In this case a Fourier sine expansion for $f_Q$ would be suitable and the formulas in the theorems in section~2 and 3 aim  at this case. Nevertheless, the condition (3) is not necessary for the validity of the representations in these theorems. Besides, they can be modified easily to apply to general Fourier series.

Theorem 1 in Royen (2007b) provides a general formula for the Laplace transform (LT) of a quadratic form $X'AX$ of  $X'=(X_1,\ldots,X_n)$ with an "$m$--factorial" matrix \mbox{$A = D \mp CC'$,} where $D>0$ is diagonal and $C$ is an $ n \times m$--matrix of rank $m$. $Q$ is a one--factorial quadratic form with the identity matrix $D$ and $C' = n^{-1/2}(1,\ldots,1)$. Since $Q$ is bounded its LT is holomorphic and its characteristic function (cf) $\widehat{f}_Q$ is obtained from the LT by the substitution $t \to -it$. With
\begin{equation}
\label{eq:4}
\psi(t,z):= \cE \exp \left(- (z+\sqrt{-it}\  X)^2\right), \ t\in \bR, \ z \in \bC,	
\end{equation}
it follows from the above formula that
\begin{equation}
\label{eq:5}
	\widehat{f}_Q(t) = \sqrt{n/\pi} \int^\infty_{-\infty} \left(\psi(t,y)\right)^n dy.
\end{equation}
This identity is also directly obtained when $\exp (w^2)$ is replaced by the integral 
\[
\pi^{-1/2} \int^\infty_{-\infty} \exp(-y^2-2wy)dy = \exp(w^2)
\]
with $w = (-it/n)^{1/2} \sum^{n}_{j=1} X_j$ in the cf of $Q = \sum^{n}_{j=1} X^2_j - (\sum^{n}_{j=1} X_j)^2/n$ after the substitution $y \to y\sqrt n$.

The Fourier sine coefficients of $f_Q$ are given by
\begin{equation}
\label{eq:6}
(2/q)\Im m(\widehat{f}_Q(t_k)), \ t_k = k\pi/q = 	\Bigg\{ \begin{matrix} 4k\pi/n & , \ n \ \mathrm{even}\\
                                4kn\pi/(n^2-1) & , \ n \ \mathrm{odd} \end{matrix} \Bigg\}.
\end{equation}
However, the formula (5) is numerically useless since
\[
\max_{y\in \bR} \left| \exp \left( - (y + x \sqrt{t/2} - ix \sqrt{t/2}\ )^2\right)\right| = \exp(tx^2/2)
\]
increases very rapidly for large $t$. Generally, integration over $\exp(-z^2)$ should be avoided in the "critical area" defined by $|\Re e(z)| < |\Im m(z)|$. This is accomplished in (5) e.g. by changing the path of integration with $t =t_k$ to 
\begin{equation}
\label{eq:7}	
\gamma_k := \left\{ y - \sqrt{-it_k}\ \big| \  y< 0\right\} \cup \left\{ u\sqrt{-it_k}\ \big| -1 \le	 u \leq 0\right\} \cup \{ y |y>0\},
\end{equation}
where $\gamma_k$ has been identified with the set of its points for simplicity. For a density $f$ on $(-1,1)$ $\gamma_k$ is defined by
\[
\{y + \mathrm{sgn}(y) \sqrt{-it_k}\  \big| \ y \in \bR \backslash \{0\}\} \cup \{ u \sqrt{-it_k} \ \big| \ -1 \leq u \leq 1\}
\]
 It should be noted that 
\begin{equation}
\label{eq:8}
	\psi^* (t,u) := \psi(t,u\sqrt{-it}) = \cE \exp (it(X+u)^2)
\end{equation}
is the cf of $(X +u)^2$.

If the density $f_Q$ is square integrable $(f_Q \in \cL^2)$ then an absolutely and uniformly convergent series for the cdf of $Q$ follows from (5) and (6) with the changed paths $\gamma_k$ by integration over $x$:
\begin{equation}
\label{eq:9}
	F_Q(x) = P\{Q\leq x\} = (2/\pi)\sum^\infty_{k=1} \Im m(\widehat{f}_Q(t_k))(1-\cos(t_kx))/k, \ 0 \leq x\leq q,
\end{equation}
where the terms free of $x$ sum up to $1 - \cE(Q)/q$, and
\begin{equation}
\label{eq:10}
\widehat{f}_Q(t_k) = \sqrt{n/\pi} \int_{\gamma_k} (\psi(t_k,z))^n dz.	
\end{equation}
The more elegant formula
\begin{equation}
\label{eq:11}
	\widehat{f}_Q(t_k) = \sqrt{n/\pi} \sqrt{-it_k} \int_{-\infty}^\infty (\psi^*(t_k,u))^n du
\end{equation}
will also be proved in section 2, but formula (10) is numerically more attractive because of the rapid decrease of $\exp(-y^2)$, $y \in \bR$.

Finally, a univariate integral representation for $F_Q$ arises if the change of summation and integration in (9) can be justified. Then the integrand is a Fourier series with simpler coefficients, and the higher partial sums of this series exhibited a very smooth appearance in all the plotted examples. The proof of this final representation for $F_Q$ is simple if the parent densities belong to the class
\begin{equation}
\label{eq:12}
	BVC[a,b] := \{f:x \in [a,b] \to \bR, \ f\in C[a,b], \ f' \in \cL^1(a,b)\}.
\end{equation}
Thus, $f \in BVC[0,1]$ means $f$ is continuous and consequently bounded on $[0,1]$ and of bounded variation. Densities from this class will be treated in section 2. If the density $f$ of $X$ is a polynomial or a trigonometrical polynomial then $\psi(t,z)$ is representable by simple finite terms, containing only the error function $\erf(z)$, $ |\Re e(z)| \geq |\Im m(z)|$, the exponential function and powers. For more general densities some different series representations of $\psi$ will be given. Section 2 also contains a generalization to the distribution of a one--factorial positive semi--definite quadratic form $Q$ of bounded independent, but not necessarily identically distributed, continuous random variables. To include also unbounded densities, e.g. the whole family of beta densities, a broader class will be investigated in section 3.

Formulas from the handbook of mathematical functions by M.~Abramowitz and I.~Stegun are cited by (A.S.) and their number. The symbol $\cE$ stands for "expectation".\\

\section*{\large 2. \ The distribution of $Q$ for bounded parent densities with an absolutely \linebreak \hspace*{8mm} integrable derivative }
\label{section:2}

\begin{theorem}
\label{theorem:1}
Let $X$ be a random variable with a density $f\in BVC [0,1]$ as defined in (12). The cdf $F_Q$ of $Q = \sum^n_{i=1} (X_i - \overline{X})^2$ from a corresponding random sample is given for all $n\geq 2$ by
\begin{equation}
\label{eq:13}
F_Q(x) =2\pi^{-3/2}	\sqrt n \ \int^\infty_{-\infty} \sum^\infty_{k=1} \Im m\left( \sqrt{-it_k}\left(\psi^*(t_k,u)\right)^n\right) (1-\cos(t_kx))/k \ du
\end{equation}
with the cf $\psi^*(t,u)$ of $(X+u)^2$ and $t_k = k\pi/q$, $q = \sup Q$ from (2), or by

\begin{eqnarray}
	\label{eq:14}
F_Q(x) &=& 2\pi^{-3/2}	\sqrt n \ \int^1_0 \sum^\infty_{k=1} \Im m\left( \sqrt{-it_k}\left(\psi^*(t_k,-u)\right)^n\right) (1-\cos(t_kx))/k \ du\nonumber\\[-2ex]
&&\\[-1ex]
&&\hspace*{-15mm} + 2\pi^{-3/2}	\sqrt n \ \int^\infty_0 \sum^\infty_{k=1} \Im m\left( \left(\psi(t_k,y)\right)^n + \left(\psi\left(t_k,-y-\sqrt{-it_k}\right)\right)^n\right) (1-\cos(t_kx))/k \ dy\nonumber
\end{eqnarray}
with $\psi(t,y) = \cE \exp\left( - \left( y + \sqrt{-it}\ X\right)^2\right)$.
\end{theorem}

The density of $Q$ is square integrable for all $n \geq 3$ and it is obtained by the derivative under the integral at least for all $n \geq 4$.

\paragraph{\sc Remarks : } For numerical evaluations the upper limit of integration $y = \infty$ in the second integral in (14) can be lowered to small values $y_0$, depending on $n$ and $f$, because of the very rapid decrease of the integrand.

If $f(1/2 + x) = f(1/2-x)$ then the integral in (13) can be replaced by $2 \int^\infty_{-1/2} \ldots du$, the first integral in (14) by $2 \int^{1/2}_0 \ldots du$, and it is
\[
\int^\infty_0 \left( \psi \left( t_k, -y-\sqrt{-it_k}\right)\right)^n dy = \int^\infty_0 \left( \psi \left( t_k, y\right)\right)^n dy.
\]

Before the proof we look at the order of magnitude of $F_Q(x)$, $1-F_Q(q-x)$, $x \to 0$, if $f$ is only supposed to be bounded. In particular, for the uniform distribution on $[0,1]$ $ F_Q(x)$ is asymptotically given by the volume of a thin cylinder whose axis goes through $(0,\ldots,0)$ and $(1,\ldots,1)$. Thus 
\[
F_Q(x) \simeq \sqrt n \ b_{n-1} x^{(n-1)/2},\ x \to 0,
\]
with the volume $b_{n-1}$ of the $(n-1)$-unit ball. There are exactly $(1+n-2[n/2])\left( n \atop [n/2]\right)$ corners of the $n$--unit cube with maximal distance $\sqrt q$ from this axis. If $x$ approaches $q$ then $1-F_Q(x)$ is asymptotically given by the volumes of small simplices with these corners at their tops, which leads to
\[
1-F_q(q-x) \simeq (1+n-2[n/2])\left(n \atop [n/2]\right) \frac 1{n!} \left( \frac n{4q}\right)^{n/2} x^n, \ x \to 0.
\]
Therefore, we obtain for any bounded parent density $f$ on $[0,1]$
\begin{equation}
\label{eq:15}
	F_Q(x) = O\left( x^{(n-1)/2}\right), \ 1 - F_Q(q-x) = O(x^n), \ x \to 0,
\end{equation}
with $O$--constants depending on $f$ and $n$.

With some additional regularity conditions for $f$  we can expect for the corresponding density $f_Q$
\begin{equation}
\label{eq:16}
f_Q(x) = O\left( x^{(n-3)/2}\right), \ f_Q(q-x) = O\left(x^{n-1}\right), \ x \to 0,	
\end{equation}
and therefore $f_Q(x) \to 0$, $x \to 0$, if $n \geq 4$. Then a Fourier--sine series for $f_Q$ is suitable.\\

The proof of theorem 1 is accomplished essentially by the following lemma:

\begin{lemma}
\label{lemma:1}
Let $f$ be a probability density within $BVC[0,1]$. Then the following bounds for $\psi(t,y)$, $\psi^*(t,u)$, $t>0$, $y >0$, are obtained:\\
\begin{eqnarray}
\label{eq:17,18,19}
|\psi(t,y)|,|\psi(t,-y-\sqrt{-it}| &\leq& Ct^{-1/2} \min(1,y^{-1} e^{-y^2}),\\
|\psi^*(t,u)|&\leq& C_1\ t^{-1/2},\  u \in \bR,\\ 
|\psi^*(t,u)|&\leq& C_2\ t^{-1}|u|^{-1}, |u|\geq 2, 
\end{eqnarray}\\
with suitable constants $C,C_1, C_2$ depending only on $f$. \\

Furthermore, with $t =v^2 t_k$, $t_k$ from (6), $v > 1$,\\
\begin{eqnarray}
\label{eq:20,21}
\int^\infty_0 \left| \psi\left( t_k, \pm \left( y + v \sqrt{-it_k}\right)\right)\right|^n dy &=& O(v^{-n}), \ v \to \infty, \\
\sqrt{t_k} \int^\infty_{-\infty} \left| \psi^*(t_k,u)\right|^n du &=& O\left(t_k^{-(n-1)/2}\right) \ ,
\end{eqnarray}
and
\begin{equation}
\label{eq:22}
\left|\widehat{f}_Q(t_k)\right| =O\left(t_k^{-(n-1)/2}\right)
\end{equation}
with $O$--constants  depending only on $f$ and $n$.
\end{lemma}

\begin{proof}
By partial integration we find

\begin{eqnarray*}
\psi(t,y) &=& \frac 12 \sqrt \pi (-it)^{-1/2}\left(\erf \left(y + \sqrt{-it}\right) - \erf(y)\right) f(1)\\
&& - \frac 12 \sqrt{i \pi/t} \ \int^1_0 \left( \erf \left(y + x\sqrt{-it}\right) - \erf(y)\right) f'(x)dx,
\end{eqnarray*}
and therefore
\begin{equation}
\label{eq:23}
|\psi(t,y)| \leq C_1 \ t^{-1/2},
\end{equation}
since $\erf(z)$ is absolutely bounded on $\{z|\ |\Re e(z)| \geq |\Im m(z)|\}$.

With $\big|\exp \big(-\big(y + x \sqrt{-it}\big)^2\big)\big| = \big| \exp\big( - \big( y + x \sqrt{t/2} - ix\sqrt{t/2}\big)^2\big)\big| \\
= \exp\big(-y^2 - xy\sqrt{2t}\big)$ we get
\begin{equation}
\label{eq:24}
|\psi(t,y)| \leq C_2 \ t^{-1/2} y^{-1} \exp(-y^2).
\end{equation}
Together with (23) this implies (17) for $\psi(t,y)$ with $C =\max (C_1,C_2)$.  \\ Since \mbox{$\psi(t, -y-\sqrt{-it}; \ f(x)) = \psi(t,y; f(1-x))$} the bound in (17) also applies to \mbox{$\psi(t,-y-\sqrt{-it})$.}

The bound in (18) is found  by the same way as in (23) with $y$ replaced by $u\sqrt{-it}$. If $|u| \geq 2$ we have $|u+x| \geq |u|/2$ and using (A.S. 7.3.22) we obtain
\[
\left| \erf \left( \sqrt{-it} (u+x)\right) - \erf\left( \sqrt{-it}\ u\right)\right| \leq C_3 \ t^{-1/2} |u|^{-1},
\]
and consequently the bound in (19).\\

For the proof of (20) we use

\begin{eqnarray*}
\lefteqn{\int^y_0 \left| \exp\left(-\left(z + \sqrt{-it} + x\sqrt{-it_k}\right)^2\right)\right|dz}\\
&&\leq\int^y_0 \exp\left(-z^2 - z \sqrt{2t} - xz\sqrt{2t_k}\right)dz < \int^\infty_0 \exp\left(-z\sqrt{2t}\right)dz
= (2t)^{-1/2},
\end{eqnarray*}\\
and consequently  
 
\begin{eqnarray*}
\lefteqn{\left|\erf \left(y + \sqrt{-it}+ x \sqrt{-it_k}\right) - \erf\left( y + \sqrt{-it}\right)\right|}\\
&=& \left|\erf \left( \sqrt{-it}+ x \sqrt{-it_k}\right) - \erf\left(\sqrt{-it})\right)\right| +O(t^{-1/2}) = O(t^{-1/2}),
\end{eqnarray*}\\
to obtain again by partial integration
\begin{equation}
\label{eq:25}
\left| \psi \left(t_k, y+v\sqrt{-it_k}\right)\right| = O(t^{-1/2}) = O(v^{-1}), \ v \to \infty,	
\end{equation}
with an $O$--constant depending on $f$ and not on $k$. Besides
\begin{eqnarray}
\label{eq:26}
\left| \psi\left(t_k, y + \sqrt{-it}\right) \right| &\leq& \int^1_0 \left| \exp\left(-\left( y + \sqrt{-it} + x \sqrt{-it_k}\right)^2\right)\right| f(x)dx\nonumber \\
&=&  \int^1_0  \exp\left(- y^2 - y \sqrt{2t}- xy \sqrt{2t_k}\right) f(x)dx \nonumber\\
&\leq& C_1 \ \exp \left( -y^2 - y \sqrt{2t}\right)\left( y \sqrt{2t_k}\right)^{-1}.
\end{eqnarray}\\

Now the integral $\int^\infty_0 |\psi|^n dy$ is splitted over the intervals \mbox{$0<y \leq	 (2t_k)^{-1/2}$} and \mbox{$(2t_k)^{-1/2} < y$.} Following from (25) and (26) the contribution of these integrals are $O(v^{-n})$ and 
\[
O\left( t_k^{-1/2} \int^\infty_1 \exp(-nvy) y^{-n}dy\right) = O\left( e^{-nv}\right), \ v \to \infty,
\]
which provides the assertion (20) for the upper sign. The proof for the lower sign is very similar.

The relation (21) follows immediately from (18) and (19). Finally,
\[
|\widehat{f}_Q(t_k)| \leq \sqrt{n/\pi} \int_{\gamma_k} |\psi(t_k,z)|^n |dz|
\]
with $\gamma_k$ from (7), and therefore the assertion (22) follows from (17) and (18). \hfill      $\Box$
\end{proof}

\paragraph{\sc Proof of theorem 1 :} The Fourier cosine coefficients of $F_Q$ are given by $-2/(qt_k)\Im m(\widehat{f}_Q(t_k))$ and with (22) from lemma 1 we obtain for all $n \geq 2$ the absolutely and uniformly convergent series for $F_Q$, already given in (9), (10). The square integrability of $f_Q$ for all $n \geq 3$ is a consequence of (22). The change of summation and integration in (9) is justified by the majorants provided by (17) and (18) which proves (14) and gives the density $f_Q$ by the derivative under the integral at least for all $n \geq 4$.

To prove (13) the paths $\gamma_k$ are modified to

\begin{eqnarray}
\label{eq:27}
\gamma_{kv} &:=& \{ y -vw_k\,|\, y <0\} \ \cup \ \{uw_k\,|\,-v\leq u\leq v\}\nonumber \\
&&\cup \ \{y +vw_k\,|\,y>0\}, \ w_k = \sqrt{-it_k}, \ v > 1.
\end{eqnarray}
Then
\[
\int_{\gamma_{kv}}(\psi(t_k,z))^n dz = w_k \int^v_{-v}(\psi(t_k,uw_k))^n du + O (v^{-n})
\]
because of (20). Therefore, equation (11) follows for $v \to \infty$. Now the representation (13) follows from the majorants obtained by (18) and (19). The density $f_Q$ is obtained again by the derivative  under the integral at least for all $n \geq 4$. \hfill $\Box$  \\

The representation in (5) for the cf $\widehat{f}_Q$ is a consequence of the "one--factor" structure of $Q$. Therefore, theorem 1 can be generalized to one--factorial quadratic forms.

\begin{theorem}
\label{theorem:2}
Let $Q = X'(D \mp cc')X$ be a one--factorial positive semi--definite quadratic form with $n$ bounded independent --- but not necessarily identically distributed --- random variables $X_j$ with densities $f_j \in BVC [a_j,b_j]$, $a_j \leq 0$, where $D = \mathrm{Diagonal} (d_1,\ldots,d_n) > 0$ and $c = (c_1,\ldots,c_n)' \in \bR^n$ with at least two components $c_j \neq 0$. With the cf $\psi^*_j (t,u)$ of $d_j(X_j \pm c_j u/d_j)^2$, $u \in \bR$, $j=1,\ldots,n$, $\lambda = 1 \mp \sum^n_{j=1} c_j^2/dj \geq 0$ and the numbers $t_k = k\pi / \sup Q$, $k \in \bN$, the cdf $F_Q$ of $Q$ is given by

\begin{eqnarray}
\label{eq:28}
\lefteqn{F_Q(x) =} \\
&& 2\pi^{-3/2} \int^\infty_{-\infty} \sum^\infty_{k=1} \Im m\bigg( \exp \left( \pm i\lambda t_ku^2\right) \sqrt{\mp it_k} \prod^n_{j=1} \psi^*_j(t_k,u)\bigg) \cdot \left( 1 - \cos (t_k x)\right) /k \ du \nonumber
\end{eqnarray}
and the density $f_Q$ is square integrable for all $n \geq 3$.
\end{theorem}

\paragraph{\sc Remarks : } The condition $a_j \leq 0$ was only introduced to get $[0, \sup Q]$ as the range of $Q$. With different intervals $[a_j, b_j]$ and indefinite forms $Q$ general Fourier series for $f_Q$ and $F_Q$ can be derived in the same way.

The corresponding generalization of (14) can also be derived if the paths $\gamma_k$ are replaced by the $\gamma_{kv}$ in (27) with a sufficiently large number $v$, depending on the $4n$ numbers $a_j, b_j, c_j, d_j$.

\paragraph{\sc Proof of theorem 2 :} The upper signs below refer to the negative sign in $D \mp cc'$. From theorem 1 in Royen (2007b) the cf $\widehat{f}_Q$ is given by 
\begin{equation}
\label{eq:29}
	\widehat{f}_Q(t) = \pi^{-1/2} \int^\infty_{-\infty} \exp \left(-\lambda y^2\right) \prod^n_{j=1} \psi_j (t,y) dy
\end{equation}
with $\psi_j(t,y) = \cE \exp \left( -d_j\left( \sqrt{-it}\ X_j + c_j y/d_j\right)^2 \right), \ t,y \in \bR$. Turning the way of integration in (29) by $\mp \pi/4$ we get
\begin{equation}
\label{eq:30}
	\widehat{f}_Q(t) = \sqrt{\mp \ it/\pi} \int^\infty_{-\infty} \exp( \pm \ i\lambda tu^2) \prod^n_{j=1} \psi^*_j(t,u) du .
\end{equation}
With the numbers $t_k$ we obtain by the same method as used for (18), (19) the estimates 

\begin{eqnarray}
\label{eq:31}
|\psi^*_j(t_k,u)| &=& O(t^{-1/2}_k), \ u \in \bR, \nonumber\\
\mbox{and} && \\
|\psi^*_j(t_k,u)| &=& O(t^{-1}_k |u|^{-1}), \ c_j \neq 0, \ |u| \geq u_0 ,\nonumber
\end{eqnarray}\\
with $O$--constants depending on $f$ and $a_j,b_j,c_j,d_j$, and a sufficiently large number $u_0$ depending on the $4n$ numbers $a_j,b_j,c_j,d_j$. Therefore
\[
|	\widehat{f}_Q(t_k)| = O\left(t_k^{-(n-1)/2}\right),
\]
and consequently $f_Q \in \cL^2 (0,\sup Q)$ for $n \geq 3$.

The Fourier cosine series for $F_Q$ from (9) with 	$\widehat{f}_Q(t_k)$ obtained from (30) is again absolutely and uniformly convergent for all $n \geq 2$ and the assertion in (28) follows by the change of summation and integration in this series, justified by the majorant derived from (31). \hfill $\Box$ \\

This section is concluded by some formulas for $\psi(t,z) = \cE \exp \left( - \left( z + \sqrt{-it}\ X\right)^2\right)$. If $X$ has a polynomial density $f$ of degree $d$ we get 
\begin{equation}
\label{eq:32}
\psi(t,z) = \frac 12 \sqrt{-i\pi/t}\ \exp(-z^2) f\Big( - \frac 12 \sqrt{i/t} \frac\partial{\partial z}\Big) \left( \exp(z^2)(\erf(z+\sqrt{-it})-\erf(z))\right)	
\end{equation}
or equivalently
\begin{eqnarray}
\label{eq:33}
\lefteqn{\psi(t,z) = \sqrt{i/t} \int^{z+\sqrt{-it}}_z \exp (-w^2)f((w-z)/\sqrt{-it})dw\nonumber}\\[-2ex]
&&\\[-2ex]
&&\textstyle{ = \sum\limits^d_{k=0} f^{(k)} \left(-z/\sqrt{-it}\right){}_1F_1 \left(\frac{k+1}2, \frac{k+3}2,-w^2\right)\left(w/\sqrt{-it}\right)^{k+1}\big/(k+1)!\Big|^{z+\sqrt{-it}}_{w=z}} \ . \nonumber
\end{eqnarray}\\

\noindent In particular,
\begin{equation}
\label{eq:34}
\psi^*(t,u) =  \sum^d_{k=0} f^{(k)} (-u){}_1F_1 \left(\textstyle{\frac{k+1}2, \frac{k+3}2},itv^2\right)v^{k+1}\big/(k+1)!\Big|^{u+1}_{v=u} \ . 
\end{equation}
Writing $z^k$ as a linear combination of Hermite polynomials $H_j$ the functions 
\[
\int^w_0 e^{-z^2} z^k dz = {}_1F_1 \left(\textstyle{\frac{k+1}2, \frac{k+3}2},-w^2\right)w^{k+1}\big/(k+1) 
\]
are representable by linear combinations of $\erf(w)$ and $H_j(w) \exp(-w^2),\  j=0,\ldots,k-1$.

A general formula for any bounded or unbounded density $f$ with the moments $\mu_j$ is given by
\begin{equation}
\label{eq:35}
\psi(t,z) = \exp(-z^2) \sum^\infty_{j=0} \mu_j H_j(z) (-\sqrt{-it})^j/j!.
\end{equation}
There is an alternative if the cf $\widehat{f} = \widehat{f}_c + i\widehat{f}_s$ of $f$ is available. With the "base functions"

\begin{eqnarray}
\label{eq:36}
\lefteqn{\psi_{2m}(t,z) := \psi(t,z; \exp(2m\pi ix))}\\
&=& \frac 12 \sqrt{i \pi/t} \ \exp\left(-im^2\pi^2/t\right) \exp\left(2m\pi z \sqrt{-i/t}\right)\erf\left(z+x\sqrt{-it} + m\pi\sqrt{-i/t}\right)\Big|^1_{x=0},  \nonumber\\
&& \hfill t>0\ ,\nonumber
\end{eqnarray}\\[1ex]
in particular with $z = u\sqrt{-it}$

\begin{eqnarray}
\label{eq:37}
\lefteqn{\psi^*_{2m}(t,u)}\\
 &=& \frac 12 \sqrt{i\pi/t} \ \exp\left(-im^2\pi^2/t\right)\exp\left(-2m\pi iu\right) \erf\left(\sqrt{-it}\ (u+x)  + m\pi\sqrt{-i/t}\right)\Big|^1_{x=0}\ ,\nonumber
\end{eqnarray}
and
\begin{eqnarray}
\label{eq:38}
\lefteqn{
\begin{array}{lcl}
\psi_{c,m}(t,z) & := & \psi(t,z;\cos(m\pi x))\\
\psi_{s,m}(t,z) & := & \psi(t,z;\sin(m\pi x))
\end{array}\quad = }\\
&\pm& \frac 14 \sqrt{\pm \ i\pi/t} \ \exp\left( - im^2 \pi^2/(4t)\right) \cdot \nonumber \\
&&\left( \exp \left(-m\pi z\sqrt{-i/t}\right) \erf\left( z+x \sqrt{-it} - \frac 12 m\pi \sqrt{-i/t}\right) \right.\pm \nonumber\\
&&\left. \exp \left(m\pi z\sqrt{-i/t}\right) \erf\left( z+x \sqrt{-it} + \frac 12 m\pi \sqrt{-i/t}\right) \right) \Big|^1_{x=0}\ , \nonumber
\end{eqnarray}
where the upper signs refer to $\psi_{c,m}$, and with the Fourier coefficients $\widehat{f}(-2m\pi)$ of $f$ it follows
\begin{equation}
\label{eq:39}
\psi(t,z;f) = \sum_{m\in \bZ} \widehat{f}(-2m\pi)	\psi_{2m}(t,z),
\end{equation}
if $f \in \cL^2(0,1)$. Similar series are obtained from a Fourier sine or Fourier cosine series for~$f$.

For unbounded $f$, and in particular for $f \notin \cL^2(0,1)$, we can use the following relation obtained by partial integration.
\begin{equation}
\label{eq:40}
	\psi(t,z;f) = \exp(-z^2) - 2z \sqrt{-it}\ \psi(t,z;1-F) + 2it \psi(t,z;(1-F)id),
\end{equation}
with $id(x) = x$ and the cdf $F$ of $X$.

The Fourier cosine series for $1- F(x)$ and the Fourier sine series for $x(1-F(x))$ have the coefficients
\begin{equation}
\label{eq:41}
	\frac 12 A_0 = \mu_1 = \cE(X), \ A_m = 2 \widehat{f}_s (m\pi)/(m\pi)
\end{equation}
and
\begin{equation}
\label{eq:42}
	b_m = 2\left( \widehat{f}_s (m\pi)/(m\pi) - \widehat{f}'_s (m\pi)\right)\Big/(m\pi), \ m \in \bN\ ,
\end{equation}
respectively, where $\widehat{f}'_s $ denotes the derivative of $\widehat{f}_s = \Im m(\widehat{f})$. Therefore

\begin{eqnarray}
\label{eq:43}
\lefteqn{\psi(t,z;f) = \exp(-z^2) -}\\
&&2z \sqrt{-it} \left( \mu_1 \psi_0(z) + \sum^\infty_{m=1} A_m \psi_{c,m}(t,z)\right) + 2 it \sum^\infty_{m=1} b_m \psi_{s,m}(t,z).\nonumber
\end{eqnarray}\\[1ex]
Similar formulas are obtained from Fourier series for $1- F(x)$ and $x (1-F(x))$ with the $ONS$ $ \{\exp (2m\pi ix)| m \in \bZ\}$ on $[0,1]$. E.g. the cf of any beta density $f(x) = x^{p-1}(1-x)^{q-1}/ B(p,q)$ is given by \mbox{$\widehat{f}(t) = {}_1F_1 (p,p+q, it)$} with the derivative $\widehat{f}'(t) = ip/(p+q)\  {}_1F_1 (p+1,p+q+1, it)$. Possibly, repetition of the step leading from (39) to (43) might provide even more rapidly converging series for $\psi$. Unevitably, values of $\psi$ with large arguments are needed but there are at least some programs available which provide rather accurate values of some well known transcendental functions for large arguments.

\section*{\large 3. \ The distribution of $Q$ for unbounded  densities  }
\label{section:3}

A first idea of the order of magnitude of $F_Q(x)$, $x\to 0$, is found by the following inequalities. With the range $R$ of the random sample $X_1,\ldots, X_n$ we get 
\[
\frac 12 \ R^2 \leq Q \leq \frac 14 \ n R^2, \ n \geq 2,
\]
and consequently
\begin{equation}
\label{eq:44}
P\{ R \leq 2(x/n)^{1/2}\} \leq  P\{ Q \leq x\} \leq P\{ R \leq (2x)^{1/2}\}\ .	
\end{equation}
At first it is useful to investigate the special density
\begin{equation}
\label{eq:45}
f_p(x) = px^{p-1}, \ 0 < x \leq 1, \ 0 < p < 1\ .	
\end{equation}

\begin{lemma}
\label{lemma:2}
\begin{equation}
\label{eq:46}
	P\{Q \leq x\} = \Bigg\{ 
	\begin{array}{ll}
	O\left( x^{min(pn,n-1)/2}\right) & , \ np \neq n-1\\
	O( x^{(n-1)/2} \ln(1/x) & , \ np = n-1
	\end{array} \Bigg\}\ , \ x\to 0 
\end{equation}
for $Q$ belonging to the parent density $f_p$ from (45).
\end{lemma}

\begin{proof}
$P\{R \leq \varepsilon\} = np \int^1_0 \left( \min (1, (x+\varepsilon)^p) - x^p\right)^{n-1} x^{p-1} dx = np (I_1 + I_2 + I_3)$, where the $I_j$  denote the subsequent integrals on $(0,\varepsilon], (\varepsilon, 1-\varepsilon), [1 - \varepsilon, 1]$. The last integral is $O(\varepsilon^n)$ and can be neglected.
\begin{eqnarray*}
I_1 &=& \varepsilon^{np} \int^1_0 \left( (1+t)^p - t^p\right)^{n-1} t^{p-1} dt = C \varepsilon^{np}.\\
I_2 &=& p^{n-1} \varepsilon^{np} I^*_2 \qquad \mbox{ with }\\
&& \int^{1/\varepsilon -1}_1 (1+t)^{(p-1)(n-1)} t^{p-1} dt \leq I^*_2 \leq \int^{1/\varepsilon -1}_1 t^{(p-1)n}  dt.
\end{eqnarray*}
Thus 
\[
\begin{array}{ll}
I^*_2 \leq  ((1-p)n-1)^{-1}\ , & n(1-p) > 1 \Longleftrightarrow np < n-1 \ ,\\
I^*_2 \leq  \ln(1/\varepsilon )\ , & np = n-1\ ,\\
I^*_2 \leq  ((p-1)n +1)^{-1} \varepsilon^{(1-p)n-1} \ , \quad & np >n-1 \ ,
\end{array}
\]
and therefore
\[
I_2 = \Bigg\{ 
	\begin{array}{ll}
	O\left( \varepsilon^{min(np,n-1)}\right) & , \ np \neq n-1\\
	O( \varepsilon^{n-1} \ln(1/\varepsilon) & , \ np = n-1
	\end{array} \Bigg\}\ .
\]
Now the assertion follows from (44) with $\varepsilon = (2x)^{1/2}$. \hfill $\Box$
\end{proof}

A more accurate asymptotic relation is also available:

\begin{theorem}
\label{theorem:3}
The cdf $F_Q$ of $Q$ belonging to the density $f_p(x) = px^{p-1}$, $0 < p < 1$, satisfies the following asymptotic relations:
\begin{equation}
\label{eq:47}
	F_Q(x) \simeq C_{pn}x^\delta , \ x \to 0 , \ \delta = \min (np, n-1)/2, \ np\neq n-1\ ,
\end{equation}
\begin{equation}
\label{eq:48}
	F_Q(x) \simeq C_{pn} x^{(n-1)/2} \ln (1/x), \ np = n-1.
\end{equation}
With 
\begin{eqnarray}
\label{eq:49}
h^\pm_p(y) &:=& p \int^\infty_0 \exp(-(x\pm y)^2) x^{p-1} dx\nonumber\\[-2ex]	
&&\\[-2ex]
&=& 2^{-p/2} \Gamma (p+1)\exp (-y^2/2)U(p-1/2, \pm y\sqrt 2)\ ,\nonumber	
\end{eqnarray}
where $U$ denotes a parabolic cylinder function, the constants $C_{pn}$ are determined by
\begin{equation}
\label{eq:50}
	\Gamma(1+\delta) \sqrt{\pi/n}\ C_{pn} = 
	\begin{cases}
	\displaystyle\int^\infty_0 \left( (h^-_p(y))^n + (h^+_p(y))^n\right)dy\ , & np < n-1\\
	(p\sqrt \pi)^n/((p-1)n+1)\ , & np > n-1 \\
	(p\sqrt \pi)^n/2 \ , & np = n-1
	\end{cases}
\end{equation}
\end{theorem}

\begin{proof}
The LT $f^*_Q$ of $Q$ is
\[
f^*_Q(t) = \sqrt{n/\pi} \int^\infty_{-\infty} (\lambda(t,y))^n \ dy
\]
with
\begin{eqnarray}
\label{eq:51}
\lambda(t, \pm y) &=& p \int^1_0 \exp \left( - (\pm y + x \sqrt t)^2\right) x^{p-1} dx\nonumber\\
&=& t^{-p/2} p \int^{\sqrt t}_0 \exp\left( - (x \pm y)^2\right) x^{p-1} dx\\
&<& t^{-p/2} h^\pm_p(y)\ . \nonumber
\end{eqnarray}
With $0 < y < (1 - \varepsilon) \sqrt t$, \ $0 < \varepsilon < 1$, \ we get for large $t$
\begin{eqnarray}
\label{eq:52}
\int^\infty_{\sqrt t}\exp (-(x-y)^2) x^{p-1} dx &\simeq& \frac 12 \sqrt\pi \ t^{(p-1)/2} \erf c(\sqrt t -y)\nonumber\\
& <& \frac 12 \ t^{(p-1)/2}\ t^{-1/2} \varepsilon^{-1} \exp(-\varepsilon^2t),
\end{eqnarray}
and therefore 
\begin{equation}
\label{eq:53}
	\lambda(t,-y) \simeq t^{-p/2} h^-_p(y), \ t \to \infty, \ 0< y < (1-\varepsilon)\sqrt t\ .
\end{equation}
From (A.S. 19.12.3) and (A.S. 13.5.1) or from (A.S. 19.10.3)
\[
\lim_{y \to \infty} y^{1-p} h^-_p(y) = p \sqrt\pi
\]
can be derived. Thus,
\begin{equation}
\label{eq:54}
	\int_0^\infty (h^-_p(y))^n dy < \infty \quad \mbox{ if } \ np < n-1\ .
\end{equation}
Obviously
\[
\lambda(t,y) \simeq t^{-p/2} h^+_p(y), \ t \to \infty \ ,
\]
\begin{equation}
\label{eq:55}
\hspace*{-2em}\mathrm{and} \\
\end{equation}
\[
\int^\infty_0 (h^+_p(y))^n dy < \infty \ .\nonumber
\]
Now we obtain from (52), (54), (55)
\begin{equation}
\label{eq:56}
f^*_Q(t) \simeq \sqrt{n/\pi} \left( \int_0^\infty \left((h^-_p(y))^n + (h^+_p(y))^n\right)dy\right) t^{-pn/2}\ , \ t \to \infty,
\end{equation}
if $ np < n-1$\ . 

To treat the case $n >np> n-1$ the parameter $\varepsilon$ in (53) is replaced by $\varepsilon(t) = t^{-1/4}$, but nevertheless it will be denoted by $\varepsilon$ for simplicity.
\[
I = \int^\infty_0 (\lambda(t,-y))^n dy
\]
is splitted into $I_1 + I_2 + I_3 + I_4$ over the intervals $(0, (1 - \varepsilon)\sqrt t)$, $[(1-\varepsilon) \sqrt t, \sqrt t]$, \mbox{ $(\sqrt t, (1 + \varepsilon) \sqrt t]$,} $((1+\varepsilon) \sqrt t, \infty)$. Then
\begin{equation}
\label{eq:57}
I_1 \simeq t^{-np/2} \int^{(1-\varepsilon)\sqrt t}_0 (h^-(y))^n dy \simeq (p\sqrt\pi)^n((p-1)n+1)^{-1} \ t^{-(n-1)/2}\ .	
\end{equation}
For $I_2$ 
\begin{eqnarray*}
\lambda\left(t,-(1-\delta)\sqrt{t}\right)&=& p \int^1_0 \exp\left( -t(x-(1-\delta))^2\right) x^{p-1}\ dx\\
&=& p(1-\delta)^p \int^{(1-\delta)^{-1}}_0 \exp\left(-\tau(\xi-1)^2\right) \xi^{p-1}\  d\xi\ ,
\end{eqnarray*}
\[
\tau =(1-\delta)^2 t, \ \xi = (1-\delta)^{-1}x, \ 0 \leq \delta \leq \varepsilon = t^{-1/4},
\]
is investigated for $t \to \infty$. With the abelian part of the Tauber theorem for Laplace transforms of measures on $(0, \infty)$, see e.g. Feller (1971), it follows
\begin{eqnarray*}
\lefteqn{p(1-\delta)^p \int^1_0 \exp(-\tau\xi^2)(1-\xi)^{p-1}d\xi}\\
&\simeq &\frac 12 \ p \int^1_0 \exp(-\tau\zeta)(1-\sqrt\zeta)^{p-1} \zeta^{-1/2} d\zeta \simeq \frac 12 \ p \sqrt \pi \ t^{-1/2}, 
\end{eqnarray*}
and
\[
p(1-\delta)^p \int^{(1-\delta)^{-1}}_1 \exp\left(-\tau(\xi-1)^2\right) \xi^{p-1}\  d\xi \simeq \frac 12\ p \sqrt \pi \ t^{-1/2} \erf\left(\delta \sqrt t\right)\ .
\]
Thus,
\[
\lambda\left(t,-(1-\delta)\sqrt{t}\right)\simeq \frac 12 \ p \sqrt \pi \ t^{-1/2} (1 + \erf(\delta \sqrt t))
\]
and 
\begin{equation}
\label{eq:58}
	I_2 \simeq O\left(\varepsilon t^{-(n-1)/2}\right) = o\left( t^{-(n-1)/2}\right)\ .
\end{equation}
Furthermore
\[
I_3 < I_2
\]
since $\lambda(t, - (1+\delta)\sqrt t) \leq \lambda (t, - (1-\delta)\sqrt t)$. With $y = (1+u) \sqrt t$ we get
\begin{eqnarray*}
I_4 &= &t^{1/2} \int^\infty_\varepsilon \left( p \int^1_0 \exp \left( -t(1-x+u)^2\right) x^{p-1}\ dx\right)^n \ du\\
&=& o\left( t^{1/2} \int^\infty_\varepsilon \exp \left(-ntu^2\right)du\right) = o\left(\exp\left( -n\sqrt t \right)\right)\ .
\end{eqnarray*}
Together with (57) and (58) it follows
\[
\int^\infty_0(\lambda(t,-y))^n dy \simeq I_1\ .
\]
Then with 
\[
\int^\infty_0(\lambda(t,y))^n dy < t^{-np/2} \int^\infty_0 (h^+_p(y))^n dy = o(I_1)
\]
we obtain
\begin{equation}
\label{eq:59}
	f^*_Q(t) \simeq \sqrt{n/\pi}\ (p\sqrt \pi)^n((p-1)n+1)^{-1} \ t^{-(n-1)/2}, \ np >n-1\ .
\end{equation}
Now from (56) and (59) the assertion (47) follows from the Tauber theorem mentioned above. The remaining case $np = n-1$ is proved in the same way as (59). \hfill $\Box$
\end{proof}

Theorem 3 suggests $f_Q \in \cL^2(0,q)$ if $\delta > 1/2$, which is proved by lemma 3 below. In view of this theorem a "too rapid" increase of the density $f$ near to its end--points might cause some problems. E.g. let be $h(y) = y^{-1}(\ln(ey))^{-2}$, $y \geq 1$, and $f(x) = h^{-1}(x) -1$. Then $f$ is a monotone decreasing probability density on $(0,1]$ with 
\[
\lim_{x\to 0} x^{\varepsilon-1}/f(x) = \lim_{y\to \infty} y^{-1} (h(1+y))^{\varepsilon-1} = 0\ ,
\]
\[
0 < \varepsilon < 1, \ \mbox{ and } \ f \notin \cL^p(0,1) \ \mbox{ for any } \ p > 1\ .
\]
Such cases and also "too strongly" oscillating densities are excluded for unbounded densities within the following class of functions.

\paragraph{\sc Definition : }
Let $BVC^*[p_1,\ldots,p_r, q_1,\ldots q_s] (0,1)$, \ $0 < p_1, \ldots,q_s < 1$, \ denote the class of real unbounded functions $f$ on $(0,1)$ for which there exists a function \mbox{$f_0 \in BVC [0,1]$,} as defined in (12), and any numbers $a_1, \ldots, a_r, b_1, \ldots,b_s$ with
\begin{equation}
\label{eq:60}
	f(x) - \sum^r_{i=1} a_ix^{p_i-1} - \sum^s_{j=1} b_j (1-x)^{q_j-1} = f_0(x)\ .
\end{equation}

All finite linear combinations of beta densities belong to such a $BVC^*$--class since every $B(p,q)$--density with $0 < p,q<1$ differs from a suitable linear combination of $x^{p-1}$ and $(1-x)^{q-1}$ only by a function from $BVC[0,1]$.

\begin{theorem}
\label{theorem:4}
Let $X_1,\ldots,X_n$ be i.i.d. random variables with a density $f \in BVC^*[p_1,\ldots, p_r, q_1,\ldots,q_s] (0,1)$. Then the density $f_Q$ of $Q = \sum^n_{i=1} (X_i-\overline{X})^2$ is square integrable for all $n$ with $2\delta = \min(np,n-1) > 1$, where $p = \min(p_1,\ldots,p_r,q_1,\ldots,q_s)$ and the cdf of $Q$ is given again by equation (14) in theorem 1 at least for all $n$ with $\delta > 1/4$. An integral representation for $f_Q$ is obtained by the derivative under the integral in (14) for all $n$ with $\delta > 1$.
\end{theorem}

\paragraph{\sc Remarks : }
A representation of $F_Q$ only by $\psi^*(t,u)$ as in (13) is not given here because it would be numerically less useful. However, it can be shown that e.g. the condition $n \min(p, 1-q) > 1$ with $q = \max (p_1,\ldots,p_r,q_1,\ldots,q_s)$ is sufficient for such a representation.

For the proof of theorem 4 we need some majorants for $\psi(t,z)$, $z \in \gamma_t$, $t >0$, for the special densities $f_p(x) = px^{p-1}$ and $f_{p1}(x) = f_p(1-x)$, \ $0 < p < 1$, \ and the resulting bounds for $|\widehat{f}_Q(t)|$. According to the three parts of the path $\gamma_t$, defined as in (7) with $t$ instead of $t_k$, we define 
\begin{eqnarray}
\label{eq:61}
\psi^+_p(t,y) &=& \cE \exp \left( - (y+\sqrt{-it}\ X)^2\right),\nonumber\\
\psi^-_p(t,y) &=& \cE \exp \left( - (y+\sqrt{-it}\ (1- X))^2\right), \ y > 0,\\
\psi^*_p(t,u) &=& \psi^+_p \left(t,u\sqrt{-it}\right),\  -1\leq u \leq 0\ ,\nonumber
\end{eqnarray}
with expectation referring to $f_p$.

The invariance of $Q$ under the transformation $X \to 1-X$ is reflected by the easily verified relations
\begin{eqnarray}
\label{eq:62}
\psi^+_{p1}(t,y) &=& \psi^-_p(t,y), \  \psi^-_{p1}(t,y) = \psi^+_{p}(t,y), \nonumber\\[-2ex]
&&\\[-2ex]
\psi^*_{p1}(t,-u) &=& \psi^*_p(t,-(1-u)), \  0 \leq u \leq 1\ . \nonumber
\end{eqnarray}
Hence, we have only to investigate the functions $\psi^\pm_p, \psi^*_p$, defined in (61).

\begin{lemma}
\label{lemma:3}
For the functions $\psi^\pm_p, \psi^*_p$, defined in (61), the following estimates hold for large $t$:
\begin{equation}
\label{eq:63}
	|\psi^\pm_p(t,y)| = \min (1, y^{-p} \exp(-y^2)) O(t^{-p/2}),
\end{equation}
\newcounter{appeqn}
\setcounter{appeqn}{1}
\vspace{-3ex}
\begin{eqnarray*}
\label{eq:64}
\hspace*{1.cm}\begin{array}[t]{l}
 |\psi^*_p(t,-u)| \\
(0 < u \leq 1) 
\end{array}
= 
\left\{ 
\begin{array}{l@{\hspace{3.4cm}}r} 
O(t^{-p/2}) & (\arabic{equation}\alph{appeqn})\\
\stepcounter{appeqn}
O(t^{-1/2} u^{p-1}), \ u > c/\sqrt t, \ p \geq 1/2 & \quad (\arabic{equation}\alph{appeqn})\\
\stepcounter{appeqn}
O(t^{-p} u^{-p}), \ u > c/\sqrt t, \ p < 1/2 & (\arabic{equation}\alph{appeqn})\\
\end{array}
\right.
\end{eqnarray*}\\
with a suitable constant $c > 1, t>c^2$ and $O$--constants depending only on $p$.

\setcounter{appeqn}{1}
\stepcounter{equation}

Furthermore, for large $t$ the cf $\widehat{f}_Q(t)$ of $Q$ is absolutely bounded by
\begin{equation}
\label{eq:65}
	\sqrt{\frac n\pi} \int_{\gamma_t} |\psi (t,z)|^n |dz| 
	= \left\{ \begin{array}{lr}
O(t^{-\delta}), \delta = \min(np,n-1)/2, & np\neq n-1, np>1 \\
O(t^{-(n-1)/2} \ln(t)),                  & np = n-1 \\
O(t^{-(np-1/2)},                          & 1/2 < np < 1\\
O(t^{-1/2} \ln(t))                       & np = 1
	\end{array} \right\}                  
\end{equation}
with $O$--constants depending only on $p$ and $n$.
\end{lemma}

\paragraph{\sc Remark : } Obviously, the exponent within the $O$--term for $np<1$ does not give the correct order of magnitude for $\widehat{f}_Q$. This is a consequence of the coarse estimation of $|\int^1_0 (\psi^*)^n du|$ by $ \int^1_0 |\psi^*|^n du$ given below. On the other hand these estimates are at least sufficient to ensure the square integrability of $f_Q$ for $np > 1$ and this is considered as a minimal requisite for a reasonable rate of convergence in the representation (14) for $F_Q$.

\paragraph{\sc Proof : } With $t > 0$, $y >0$ and
\begin{eqnarray}
\label{eq:66}
|\exp (-(y+x \sqrt{-it})^2)| &= & |\exp (-(y + x \sqrt{t/2} - ix\sqrt{t/2})^2)|\nonumber\\
&=& \exp (-y^2 - xy \sqrt{2t})
\end{eqnarray}
we get
\begin{eqnarray}
\label{eq:67}
|\psi^+_p(t,y)| &\leq & \exp (-y^2) p\gamma\big(p,y \sqrt{2t}\big) \big(y \sqrt{2t}\big)^{-p}\nonumber\\
&<& 2^{-p/2} \Gamma(p+1)t^{-p/2} y^{-p} \exp(-y^2).
\end{eqnarray}
To obtain an estimate for small $y$ and $t>1$ the integral $\psi^+_p(t,y)$ is decomposed into two integrals $I_1$, $ I_2$ over the intervals $(0,t^{-1/2}]$, $(t^{-1/2}, 1]$ respectively. Obviously we have
\[
|I_1| \leq t^{-p/2}.
\]
For $I_2$ we find after partial integration
\begin{eqnarray*}
I_2 &=& \frac 12 \ p \sqrt \pi \ \sqrt{i/t}\left(\erf \left(y + x\sqrt{-it}\right) - \erf(y)\right)x^{p-1}\big|^1_{x=t^{-1/2}} \\
&& - \frac 12 \ p \sqrt \pi \sqrt{i/t} \int^1_{1/\sqrt t} \left(\erf \left(y + x\sqrt{-it}\right) - \erf(y)\right)(p-1) x^{p-2} \ dx\ .
\end{eqnarray*}
With the number $\max \{ \erf (z) \ \big|\  |\Re e(z)| \geq |\Im m(z)|\}$ and suitable constants $C_i, i=1,2,\ldots,$ only depending on $p$, it follows
\[ 
|I_2| \leq C_1 \ t^{-1/2} + C_2\  t^{-p/2},
\]
and therefore
\[
| \psi^+_p (t,y)| \leq C_3 t^{-p/2}, \ t > 1\ .
\]
Together with (67) this yields (63) for $\psi^+_p$.

To find a majorant for $\psi^-_p$ we proceed as before: 
\begin{eqnarray*}
|\psi^-_p(t,y)| &\leq & p \int^1_0 \exp \big( -y^2 - y \sqrt{2t} + xy\sqrt{2t}\big)x^{p-1} \ dx\\
&=&  \exp \big( -y^2 - y\sqrt{2t}\big) {}_1F_1 \big(p,p+1,y\sqrt{2t}\big)\ .
\end{eqnarray*}
With (A.S. 13.5.1) and $y\sqrt t \geq c$, where $c$ is a suitable constant, it follows
\begin{eqnarray}
\label{eq:68}
|\psi^-_p(t,y)| &\leq & C_1 \ \exp (-y^2)\left( y \sqrt{2t}\right)^{-p} \left( \Gamma (p+1)\exp\left( -y \sqrt{2t}\right)+ p\left( y \sqrt{2t}\right)^{p-1}\right) \nonumber\\
&\leq & C_2 \ y^{-p} \exp (-y^2) t^{-p/2}.
\end{eqnarray}\\
On the other hand we obtain after partial integration\\
\begin{small}
\begin{eqnarray*}
\lefteqn{\psi^-_p (t,y) =  - \frac 12 \ p \sqrt{i \pi/t} \left( \erf(y) - \erf\left(y + \sqrt{-it}\right)\right) }\\
&&- \frac 12 \ p\sqrt\pi \int^1_0 \left( - \sqrt{-it}\right)^{-1} \left(\erf \left( y + \sqrt{-it}(1-x)\right) - \erf \left( y + \sqrt{-it}\right)\right) (p-1) x^{p-2} \ dx .
\end{eqnarray*}
\end{small}

The last integral is splitted again into $I_1 + I_2$ over $x \leq t^{-1/2}$ and $x > t^{-1/2}$, $t>1$, and we find
\begin{small}
\begin{eqnarray*}
|I_1| &\leq & \int^{1/\sqrt t}_0 \Big|(-x\sqrt{-it})^{-1}\left( \erf(y+\sqrt{-it}(1-x))-\erf(y+\sqrt{-it})\right)\Big|(1-p )x^{p-1}\ dx\\
&\leq& C_3 \ t^{-p/2},\\
|I_2| &\leq & C_4 \ t^{-1/2} \int^1_{1/\sqrt t} (1-p)x^{p-2} \ dx \ \leq \ C_4 \ t^{-p/2} ,
\end{eqnarray*}
\end{small}
and therefore 
\[
|\psi^-_p(t,y)| \leq C_5 \ t^{-p/2}, \ t > 1.
\]
Together with (68) this implies (63). By the same way (partial integration and splitting the integral) the assertion (64a) is obtained.

Now a bound is given for $\psi^*_p$ containing $u$. After the substitution $x \to x/\sqrt t$ we get
\begin{equation}
\label{eq:69}
	\psi^*_p (t,-u) = p  \exp (itu^2) t^{-p/2}(I_1 - I_2)
\end{equation}
with
\[
I_1 = \int^\infty_0 \exp(-2ixu\sqrt t) \exp(ix^2)x^{p-1}\ dx
\]
and $I_2$ with the same integrand over $(\sqrt t,  \infty)$. 
The first integral is evaluated by a combination of the parabolic cylinder functions $U(-p-1/2, \ \pm(1+i)u\sqrt t)$, \mbox{$U(-p-3/2, \ \pm(1+i)u\sqrt t)$.}
By means of (A.S. 19.12.1), (A.S. 19.12.3) and (A.S. 13.5.1) (or only empirically by plotting) it can be shown that 
\begin{equation}
\label{eq:70}
	|I_1(-2u\sqrt t)| \leq C_1 (u\sqrt t)^{-\min (p,1-p)}, \ u\sqrt t \geq c\ ,
\end{equation}
with a suitable constant $c > 1$. By partial integration of $\exp(itu^2) I_2$ we obtain
\begin{eqnarray*}
|I_2 |&\leq& \int^\infty_{\sqrt t} \Big| \erf\left(\sqrt{-i}(x-u\sqrt t)\right) - \erf\left(\sqrt{-it}(1-u)\right)\Big|(1-p)x^{p-2}\ dx\\
&\leq & C_2\  t^{(p-1)/2},
\end{eqnarray*}
and together with (69) and (70) it follows
\[
|\psi^*_p(t,-u)| \leq C_2\ t^{-1/2} + C_1 \ t^{-p/2}(u\sqrt t)^{-\min(p,1-p)},
\]
which implies (64b) and (64c).

Finally, the cf $\widehat{f}_Q$ is absolutely bounded by
\[
\sqrt{n/\pi}\left(\sqrt t \int^1_0 |\psi^*_p (t,-u)|^ndu + \int^\infty_0 (|\psi^-_p(t,y)|^n + |\psi^+_p(t,y)|^n\right)dy.
\]
Due to (63) and (64a) the second integral and $\sqrt t \int^{c/\sqrt t}_0 |\psi^*_p|^n du$ are $O(t^{-pn/2})$--terms. \linebreak If $p \geq 1/2$ then\\
\begin{eqnarray}
\label{eq:71}
\sqrt t \int^1_{c/\sqrt t} |\psi^*_p|^n \ du &=& O \Big( t^{-(n-1)/2} \int^1_{c/\sqrt t} u^{(p-1)n}\ du \Big)\nonumber\\
&=& \begin{cases}
O\left( t^{-\min(np,n-1)/2}\right)\ , & np \neq n-1\\
O\left( t^{-(n-1)/2}\ln(t) \right)\ , & np = n-1\ . 
\end{cases}
\end{eqnarray}\\
If $p < 1/2$ then $pn < n-1$ and\\
\begin{eqnarray}
\label{eq:72}
\sqrt t \int^1_{c/\sqrt t} |\psi^*_p|^n \ du &=& O \Big( t^{1/2-np} \int^1_{c/\sqrt t} u^{-np}\ du \Big)\nonumber\\
&=& \begin{cases}
O \left( t^{-np/2}\right) + O(t^{-(np-1/2)}\ , & np \neq 1\\
O \left( t^{-1/2} \ln(t)\right) \ , & np = 1\ .
\end{cases}
\end{eqnarray}\\
If $np > 1$ then $\min (np/2, np-1/2) = np/2$ and (65) follows from (71) and (72). \hfill $\Box$ \\

Additionally, for the proof of theorem 4 the following lemma is applied.

\begin{lemma}
\label{lemma:4}
Let be $f_i \in \cL^1(0,1)$, $i = 1,2,$
\[
\psi_i(t,z) = \int^1_0 \exp\Big( -\big(z+x \sqrt{-it}\big)^2\Big) f_i(x)\ dx, \ t \in \bR, \ z \in \bC\ ,
\]
$\{\gamma_k |k \in \bN\}$ the family of paths defined in (7),
\[
\phi_i(t_k) := \int_{\gamma_k} |\psi_i(t_k,z)|^n\ |dz| \qquad \mathrm{and} \qquad \phi(t_k) := \int_{\gamma_k} (|c_1 \psi_1| + |c_2 \psi_2|)^n \ |dz|
\]
with any numbers $c_1,c_2$. Then the assumption $\sum_k \phi^2_i(t_k) < \infty$, $i = 1,2$, entails $ \sum_k \phi^2(t_k) < \infty$\ .
\end{lemma}

\paragraph{\sc Proof : } It has only to be shown  that $\sum_k \phi^2_{mn}(t_k) < \infty$ for the numbers 
\[
\phi_{mn}(t_k) = \int_{\gamma_k} |\psi_1|^m |\psi_2|^{n-m} |dz|, \ m=1,\ldots,n-1\ .
\]
From H\"older's inequality it follows
\[
\phi_{mn}(t_k) \leq (\phi_1(t_k))^{m/n} (\phi_2(t_k))^{(n-m)/n}\ ,
\]
and again with H\"older's inequality for sequences we get
\[
\sum_k \phi^2_{mn}(t_k) \leq \left( \sum_k \phi^2_1(t_k)\right)^{m/n} \left( \sum_k \phi^2_2(t_k)\right)^{(n-m)/n} < \infty\ . \hfill  \Box
\] 

\paragraph{\sc Proof of theorem 4 : } ~\newline
For a probability density $f \in BVC^* [p_1,\ldots,p_r, q_1,\ldots, q_s] (0,1)$ the corresponding function $\psi (t,z) = \cE \exp \left( - \left( z + x \sqrt{-it}\right)^2\right)$ is a linear combination
\[
\psi_0 + \sum^r_{i=1} a_i\psi_i + \sum^s_{j=1} b_j \psi_{r+j} = \sum^{r+s}_{i=0} c_i\psi_i
\]
with $\psi_0 $ belonging to $f_0 \in BVC [0,1]$.

The pairs $\left( \sum^m_{i=0} c_i\psi_i, c_{m+1} \psi_{m+1}\right)$, $m=0,\ldots,r+s-1$ satisfy the assumptions of lemma 4 if $2 \delta = \min(np, n-1) > 1$ with $p = \min(p_1,\ldots,p_r,q_1,\ldots,q_s)$, and we obtain
\[
\sum_k \phi^2(t_k) < \infty
\]
for the numbers
\[
\phi(t_k) = \int_{\gamma_k} |\psi(t_k,z)|^n \ |dz|
\]
and a fortiori
\[
\sum_k \left|\widehat{f}_Q (t_k)\right|^2 < \infty,
\]
which proves the square integrability of $f_Q$ if $\delta > 1/2$. 

Furthermore, because of (65) in lemma 3, the series
\[
(2/\pi) \sum^\infty_{k=1} \Im m\left( \widehat{f}_Q (t_k)\right) (1-\cos(t_kx))/k
\]
is absolutely and uniformly convergent at least for all $n$ with $\delta > 1/4$. The change of summation and integration is justified by a majorant for the integrand derived from (63) and (64) in lemma 3 and repeated use of H\"older's inequality. For the same reason the density $f_Q$ is obtained by the derivative under the integral for all $n$ with $\delta > 1$. \hfill $\Box$ \\

The last theorem provides only a coarse estimation for the order of magnitude of $F_Q(x)$ for small $x$ under rather general assumptions for the parent density $f$.

\begin{theorem}
\label{theorem:5}
Let $f$ be any probability density on $(0,1)$,
\[
r := \sup \{p | f \in \cL^p(0,1)\} < \infty \quad \mbox{and} \quad n (1-1/r) > 1\ .
\]
Then
\[
F_Q (x) = o(x^\delta), \ x \to 0,
\]
at least for all $\delta < \delta_r := (n(1-1/r)-1)/2$.
\end{theorem}

\paragraph{\sc Proof : } If $n(1-1/r) > 1$ then there exists a $p < r$ with $n(1-1/p) >1$. Then with $q = p/(p-1)$ and H\"older's inequality we obtain for
\[
\lambda(t,y) = \cE \exp (-(y + x\sqrt t)^2)
\]
the bound 
\[
\lambda(t,y) = \|f\|_p \left( \int^1_0 \exp(-q(y + x\sqrt t)^2)dx\right)^{1/q}.
\]
It follows for the LT $f^*_Q$ of $Q$\\
\begin{eqnarray*}
\lefteqn{f^*_Q(t) = \sqrt{n/\pi} \int^\infty_{-\infty} (\lambda(t,y))^n \ dy }\\
&\leq&\sqrt{n/\pi}\ \|f\|_p^n \left(\frac12 \sqrt \pi\right)^{n/q} \int^\infty_{-\infty} \left(\left(\erf\left( \sqrt q( y + \sqrt t)\right) - \erf\left( \sqrt q \ y\right)\right)\big/ \sqrt{qt}\right)^{n/q}dy\\
&=& O(t^{-\delta_p}), \ t \to \infty, \ \delta_p = (n/q - 1)/2 = (n(1-1/p)-1)/2
\end{eqnarray*}\\
with an $O$--constant depending on $p,n$ and $f$. Therefore
\[
F_Q(x) = O(x^{\delta_p}), \ x \to 0\ ,
\]
and consequently
\[
F_Q(x) = o(x^{\delta_p}), \ x \to 0, \quad \mbox{for all} \ p \ \mbox{with} \ \delta_p < \delta_r\ . \hfill \Box
\]

Unfortunately, without any further assumptions we cannot conclude from theorem 5 to $f_Q \in \cL^2$ for $n(1-1/r) > 2$.  Good estimates for $|\widehat{f}_Q(t)|$ are not easily found from the only assumption $f \in \cL^p(0,1)$, $p < r$, and it seems to be difficult to determine the lowest sample size $n$ for which e.g. $f_Q \in \cL^2$. This was the reason for the introduction of the classes $BVC^* [p_1,\ldots,p_r, q_1,\ldots,q_s](0,1)$. 

As an example we look again at the density $f_\alpha(x) = \alpha x^{\alpha-1}$, $0<\alpha<1$, belonging to $\cL^p(0,1)$ for all $p<1/(1-\alpha)$. For $1<\alpha n < n-1$
\[
F_Q(x) = O(x^{\alpha n/2}), \ x \to 0\ ,
\]
 follows from (46) or theorem 3. Theorem 5 yields only
 \[
 F_Q(x) = o(x^\delta), \ \delta < (\alpha n -1)/2.
 \]
 This discrepany is caused by the weak assumptions for $f$ in theorem 5 and by the coarse estimation with H\"older's inequality.

\references


Abramowitz, M. and Stegun, I. (1968) \textit{Handbook of Mathematical Functions}, Dover Publications Inc., New York.

Feller, W. (1971) \textit{An Introduction to Probability Theory and Its Applications}, Vol II, 2nd ed. John~Wiley \& Sons, New York. 

Mudholkar, G.S. and Trivedi, M.C. (1981) A Gaussian Approximation to the Distribution of the Sample Variance for Nonnormal Populations, \textit{Journal of the American Statistical Association} \textbf{76}, 479--485.

Royen, T. (2007a) Exact Distribution of the Sample Variance from a Gamma Parent Distribution, \mbox{arXiv:0704.0539} $\left[\textrm{math.ST}\right]$

Royen, T. (2007b) On the Laplace Transform of Some Quadratic Forms and the Exact Distribution of the 
Sample Variance from a Gamma or Uniform Parent Distribution, \mbox{arXiv:0710.5749} $\left[\textrm{math.ST}\right]$

\end{document}